\documentclass{ws-procs9x6}
\usepackage[T1]{fontenc}
\usepackage{mathtools}
\usepackage{enumerate}
\usepackage{dsfont} 
\usepackage{bm}
\usepackage{hyperref}

\newcommand{\reals}{\mathds{R}}
\newcommand{\naturals}{\mathds{N}}

\newcommand{\setsep}{\;|\;}
\newcommand{\ddd}{\,\mathrm{d}}

\DeclareMathOperator{\diam}{diam}

\newcommand{\dist}{\ensuremath{\operatorname{dist}}}

\newcommand{\ind}[1]{\mathds{1}_{#1}}
\DeclareMathOperator{\Div}{div}
\DeclareMathOperator{\Lip}{Lip}
\DeclareMathOperator*{\essinf}{ess\,inf}
\DeclareMathOperator*{\esssup}{ess\,sup}

\newcommand{\ve}[1]{\mathbf{#1}}
\newcommand{\vx}{\ve{x}}
\newcommand{\vy}{\ve{y}}

\newcommand{\vn}{\ve{n}}
\newcommand{\vm}{\ve{m}}

\newcommand{\vu}{\ve{u}}

\newcommand{\va}{\ve{a}}
\newcommand{\vb}{\ve{b}}

\newcommand{\vf}{\ve{f}}

\newcommand{\vphi}{\boldsymbol{\varphi}}

\def\softd{{\leavevmode\setbox1=\hbox{d}%
		\hbox to 1.05\wd1{d\kern-0.4ex{\char039}\hss}}}

\newcommand{\mat}[1]{\mathbf{#1}}
\newcommand{\matA}{\mat{A}}

\usepackage{color}
\usepackage{probability}
\usepackage{ stmaryrd }

\newcommand{\Dt}{\partial_t}

\newcommand{\jump}[1]{\left\llbracket#1\right\rrbracket}
\newcommand{\tildejump}[1]{\widetilde{\jump{#1}}}
\newcommand{\ava}[1]{\ensuremath{\hspace{-0.4mm}\left\{\!\!\left\{#1\right\}\!\!\right\}\hspace{-0.4mm}}}                                      
\newcommand{\fac}{\ensuremath{\mathcal{E}}}
\newcommand{\inFac}{\fac^\text{int}}
\newcommand{\exFac}{\fac^\text{ext}}
\newcommand{\sigsum}{\ensuremath{\sum_{\sigma \in \mathcal{E}_\text{int} }}}

\DeclareMathOperator{\UP}{UP}
\DeclareMathOperator{\conv}{conv}
\newcommand{\Qh}{\mathcal{Q}_h}
\newcommand{\flux}{\ensuremath{F_n^\text{up}}}

\DeclareMathOperator{\len}{length}

\usepackage{enumitem}
\definecolor{jgured}{rgb}{0.754,0.0,0.164}

\newcommand{\cb}{}

\begin{document}
	
	\title{Random compressible Euler flows}
	
	\author{M.~Luk\'a\v{c}ov\'a-Medvi{\softd}ov\'a}
	
	\address{Institute of Mathematics\\
		Johannes Gutenberg University Mainz, Germany\\
		E-mail: lukacova@uni-mainz.de\\
		www.numerik.mathematik.uni-mainz.de}
	
	\author{S. Schneider}
	
	\address{Institute of Mathematics\\
		Johannes Gutenberg University Mainz, Germany\\
		E-mail: simon.schneider@uni-mainz.de}

	\begin{abstract}
			We propose a finite volume stochastic collocation method for the random Euler system. We rigorously prove the convergence of random finite volume solutions under the assumption that the discrete differential quotients remain bounded in probability. Convergence analysis combines results on the convergence of a deterministic FV method with stochastic compactness arguments due to Skorokhod and Gy\"ongy-Krylov.  
	\end{abstract}
	
	
	\bodymatter
	
	\section{Introduction}
	The random compressible Euler equations of gas dynamics arise in many practical applications, such as meteorology, physics, engineering or medicine. In practice, model data are typically uncertain since they arise from measurements and may be influenced by various errors. Consequently, data uncertainties propagate and lead to a random PDE system. Several numerical methods have been developed in the literature to approximate random  PDE equations.  
	
	The Monte Carlo method is often used, but may be very expensive due to its slow convergence
	and large number of required samples. Alternatively, stochastic spectral methods, such as the
	stochastic Galerkin and stochastic collocation methods, are used to approximate random PDE systems efficiently. While the stochastic Galerkin
	method is intrusive, the stochastic collocation method is nonintrusive. It only requires the application of a deterministic numerical scheme 
	at certain collocation nodes.  We refer to the monographs by Le Maître and Knio \cite{Knio}, Pettersson et al. \cite{Pettersson}, Xiu \cite{Xiu}, 
	Zhang and Karniadakis \cite{Zhang}.
	
	Rigorous convergence analysis of these uncertainty quantification methods typically requires 
	uniqueness and continuous dependence of solutions on random parameters. However, continuity with respect to 
	the random parameter may be a rather strong assumption for hyperbolic problems, and we only require Borel measurability of the data $\rightarrow$ solution mapping. 
	Convergence of the stochastic collocation method for random elliptic and parabolic equations was studied, e.g., in  Babu\v{s}ka et al. \cite{Babuska}, 
	Nobile et al. \cite{Nobile}, Tang and Zhou \cite{Tang}. 

 In this paper, we study a stochastic collocation finite volume method applied to the \textit{random compressible Euler system}. As shown in our recent work Chertock et al.~\cite{CHKL} global statistical spectral methods may not be suitable for random hyperbolic conservation laws since discontinuities usually propagates also in the random direction. Global interpolation methods then yield oscillations on discontinuities due to the Gibbs phenomenon.  Therefore, we use a stochastic collocation method that works with a piecewise continuous approximation in the deterministic and the random space. We aim to rigorously prove the convergence of the stochastic collocation finite volume method. To this end, we combine deterministic analysis of a finite volume method with stochastic compactness arguments. We refer to our recent work Feireisl and Luk\'a\v{c}ov\'a~\cite{Feireisl2023}, where similar arguments have been used for the random Navier-Stokes equations, see also Feireisl et al.~\cite{MC1, MC2} for the error analysis of the Monte Carlo finite volume method. 

We start with a deterministic model.
The Euler equations of gas dynamics describe  the conservation of mass, momentum and energy
	\begin{align}\label{eq:EulerSystem}
	\partial_t \varrho + \Div (\vm) &= 0,
	\notag\\
	\partial_t \vm + \Div_x \left( \frac{\vm \otimes \vm}{\varrho} \right) + \nabla p &= 0,
	\notag\\
	\partial_t E + \Div ((E+p)\vu) &= 0.
\end{align}
The dependent variables $\varrho$, $\vm$ and $E$ denote the density, momentum and energy of the fluid, respectively.
The pressure $p$, temperature $\vartheta$, velocity $\vu$, internal energy $e$ and entropy $s$ are given by
\begin{equation*}
	p = \varrho \vartheta = (\gamma -1)\varrho e,\
	\vu = \vm/\varrho,\
	e = \varrho^{-1}(E-|\vm|^2/(2\varrho)),\
	s = \log(\vartheta^{c_v}/\varrho).
\end{equation*}
The adiabatic coefficient is $\gamma\in (1,\infty)$ and the heat capacity at constant volume is $c_v = (\gamma-1)^{-1}.$ 
 System \eqref{eq:EulerSystem} is considered on a time-space cylinder $ (0,T) \times Q$, where $T>0$ is a final time and $Q \subset \Bbb R^d$, $d=2,3$, a physical domain. 
 System (\ref{eq:EulerSystem}) is equipped with the initial data $(\varrho_0, \vm_0, E_0)$ and  impermeability boundary condition $\vu \cdot \vn = 0$ on $(0,T) \times \partial Q.$

We say that $(\varrho, \, \vm, \,E) \in C^1([0,T]\times \overline{Q}, \reals^{d+2})$ is a classical solution of (\ref{eq:EulerSystem}) if (\ref{eq:EulerSystem}) is satisfied pointwise. 
We will work with the following notion of \emph{generalized weak solution}.

\begin{definition}  [Weak solution]
	We call a tuple $(\varrho, \, \vm, \,E) \in $ \linebreak 
$L^\infty((0,T); L^\gamma(Q) \times L^{2 \gamma/(\gamma+1)}(Q;\, \reals^d)\times L^1(Q)) $ a \textit{weak solution} of (\ref{eq:EulerSystem}) with initial data $(\varrho_0, \, \vm_0, \,E_0) \in L^\gamma(Q) \times L^{2 \gamma/(\gamma+1)}(Q;\, \reals^d)\times L^1(Q)$\\
$$
\varrho_0 > 0, \ E_0 - \frac{1}{2} \frac{| \vm_0|^2 }{\varrho_0} > 0\ \mbox{ for  a.a. } \vx \in Q 
$$
	if 
$$
\varrho > 0,\   E - \frac{1}{2} \frac{| \vm|^2 }{\varrho} > 0  \mbox{ a.e.   in }  [0,T]\times Q
$$ and
	\begin{itemize}
		\item for every $\varphi \in C^\infty([0,T]\times \overline{Q})$ and every $\tau \in [0,T]$
		\begin{equation*}
			\left[\int_Q \varrho \varphi \ddd \vx \right]_{t=0}^{t=\tau} = \int_{[0,\tau]\times Q} \varrho\,\partial_t \varphi + \vm\cdot \nabla \varphi\, \ddd (t,\vx);
		\end{equation*}
		
		\item for every $\vphi \in C^\infty([0,T]\times \overline{Q}, \reals^d)$ with $\vphi\cdot \vn = 0$ a.e. on $[0,T]\times \partial Q$ and every $\tau \in [0,T]$
		\begin{align*}
			&\left[\int_Q \vm \vphi \ddd \vx \right]_{t=0}^{t=\tau} 
			\\
			&
			= \int_{[0,\tau]\times Q} \vm\,\partial_t \vphi + \frac{\vm \otimes \vm}{\varrho} : \nabla \vphi + p \Div_x \vphi \, \ddd (t,\vx),
		\end{align*}
		where $p=p(\varrho, \vm, E) = (\gamma -1) \left(E - \frac{1}{2} \frac{| \vm|^2 }{\varrho} \right)$;
		\item for every $\varphi \in C^\infty([0,T]\times \overline{Q})$ with $\varphi\geq 0$, any $\chi\in C^1(\reals)$ nondecreasing, concave and bounded from above and every $\tau \in [0,T]$
		\begin{equation*}
			\left[\int_Q \varrho \chi(s)\varphi \ddd \vx \right]_{t=0}^{t=\tau} \geq \int_{[0,\tau]\times Q} \varrho \chi(s) \,\partial_t \varphi + \varrho \chi(s) \vu \cdot \nabla \varphi\, \ddd (t,\vx),
		\end{equation*}
		where $s=s(\varrho, \vm, E) = \frac{1}{\gamma - 1}\log \left((\gamma - 1) \frac{1}{\varrho}\left( E - \frac{1}{2} \frac{| \vm|^2 }{\varrho}   \right) \right) - \log (\varrho);$
		\item for every $\tau \in [0,T]$
		\begin{equation*}
			\left[\int_Q E \ddd \vx \right]_{t=0}^{t=\tau} \leq 0.
		\end{equation*}
	\end{itemize}
\end{definition} 
Note that this notion of weak solution matches the consistency formulation of the viscous finite volume scheme, cf. \cite[Theorem 10.3]{feireisl2021Book}. 
A 
Lipschitz continuous tuple $(\varrho, \, \vm, \,E)\colon [0,T] \times \overline{Q} \to \reals^{d+2}$ which satisfies (\ref{eq:EulerSystem}) pointwise 
almost everywhere is called a \textit{strong solution} of (\ref{eq:EulerSystem}). 
It is straightforward to check that every Lipschitz continuous weak solution with strictly positive density and temperature is in fact a strong solution. Moreover, the strong solution is unique.
	
	\begin{theorem}[Uniqueness of strong solutions]\label{lem:uniquenessStrongSolutions}
		If $(\varrho_1, \, \vm_1, \,E_1)$ and $(\varrho_2, \, \vm_2, \,E_2)$ are strong solutions of (\ref{eq:EulerSystem}) with the same initial data, then $(\varrho_1, \, \vm_1, \,E_1)=(\varrho_2, \, \vm_2, \,E_2)$.
	\end{theorem}

	\begin{proof}
		Only straightforward changes of the proof of Theorem 6.2 in Ref.~\citenum{feireisl2021Book} are necessary. Note that strong solutions satisfy the equations almost everywhere. Thus, the class of admissible test functions in the weak formulation is large enough to enable testing with another strong solution. Therefore we are not restricted to domains $Q$ with boundary of class $C^2$.
	\end{proof}

\section{Discretization of the deterministic problem}
We proceed with the description of a numerical method for the deterministic Euler equations \eqref{eq:EulerSystem} and assume that the domain $Q \subset \Bbb R^d$ is open, connected, bounded and has a Lipschitz boundary.

\subsection{The viscous finite volume method}

For the space discretization, we use the numerical scheme introduced in Ref.~\citenum{feireisl2020finite}, now referred to as \textit{viscous finite volume (VFV) method}\cite{feireisl2021Book}. The following standard notation is used.

We will consider a sequence of meshes  
with increasing mesh resolution.
For simplicity, we assume that each mesh consists of shape-regular triangles in 2D and tetrahedra in 3d.
We roughly follow the space discretization presented in \cite[Section 7.2.1]{feireisl2016book}. The set of all elements (triangles or tetrahedra) of the $n$-th triangulation is denoted by $\mathcal{T}_n$.
For  $K\in \mathcal{T}_n$ the set of boundary elements of $K$ is denoted \fac(K).
The set of all boundary elements of the $n$-th mesh is denoted as $\fac_n = \cup \{\fac(K) \setsep K\in \mathcal{T}_n\}$. The exterior boundary elements are given by $\exFac_{n} = \{ \sigma \in \fac_n \setsep \sigma \subset \partial Q \}$. Likewise, the set of all interior boundary elements is given by $\inFac_{n} =\fac_n \setminus \exFac_{n}$. We assume that there are no hanging nodes, i.e. for each $\sigma \in \inFac_{n}$ there are exactly two elements $K,L\in \mathcal{T}_n$ such that $\sigma\in \fac(K) \cap \fac(L)$. 
Each boundary element $\sigma \in \inFac_{n}$ is associated with an arbitrary fixed normal vector $\vn=\vn_\sigma$. For $\sigma \in \exFac_{n}$ we fix $\vn$ to denote the outer normal unit vector of $Q$.

Further, we assume that
there exist constants $0<c_1\leq C_1$, $0<c_2\leq C_2$ such that for all $K\in\mathcal{T}_n,\, \sigma\in \fac_n,  n \in \naturals$
\begin{equation*}
	c_1 h_n^d \leq |K| \leq C_1 h_n^d, 
	\qquad c_2 h_n^{d-1} \leq |\sigma| \leq C_2 h_n^{d-1}.
\end{equation*}
Here, $h_n \coloneqq \max_{K\in\mathcal{T}_n} \diam(K)$ is assumed to vanish for $n\to\infty$.
We also assume that computational domain $Q_n = (\cup\mathcal{T}_n )^\circ$ coincides with the physical domain $Q$ for all $n\in \naturals$.
Additionally, following \cite[Definition 1]{feireisl2021Book}, we suppose that there exist control points $\vx_K$ associated to each $K\in\mathcal{T}_n$ such that for any $\sigma = K\cap L \in \inFac_{n}$, $K,L\in\mathcal{T}_n$, $\vx_K-\vx_L = \alpha \vn_\sigma$ with $\alpha \in \reals$.

We denote  by $Q_n(Q, \reals^k)$ the corresponding set of all piecewise constant functions and set
$Q_n(Q) \coloneqq Q_n(Q, \reals)$. 
We define the value of $\va_n\in \mathcal{Q}(Q;\reals^k)$ on $K\in\mathcal{T}$ by piecewise constant projection
\begin{equation*}
	\va_n^K = \frac{1}{|K|} \int_K \va(\vy) \ddd \vy.
\end{equation*} 
Further, for $\sigma\in\fac_n$ and $\va_n\in Q_n(Q, \reals^k)$ we define 
the average and jump of $\va_n$ in $\vx\in\sigma\in \fac_n$ by
\begin{equation*}
	\ava{\va_n}= \frac12 \left(\va_n^\text{out} + \va_n^\text{in} \right),
	\qquad
	\jump{\va_n} = \va_n^\text{out}- \va_n^\text{in},
\end{equation*}
respectively. Here $\va_n^\text{out,in}$ are the outward/inner limits in the normal direction $\vn_\sigma$.
The test functions in the momentum equation will be taken from the following subspace of piecewise constant functions
\begin{equation*}
	\mathcal{D}_n(\vm )\coloneqq \left\{ \vphi \in Q_n(Q, \reals^d)\setsep \vphi^\text{in} \cdot \vn_\sigma = 0  \text{ on all }\sigma \in \exFac_{n}  \right\}.
\end{equation*}
The discretization of the convective terms is based on the upwind flux
\begin{equation*}
	\UP[\va_n, \vu_n] \coloneqq \va_n^\text{in} [\ava{\vu_n}\cdot \vn]^+ + \va_n^\text{out} [\ava{\vu_n}\cdot \vn]^- .
\end{equation*}
Here, $[\cdot]^+$, $[\cdot]^-$ denote the positive and negative parts, respectively. The numerical flux function for the convective terms is augmented by an additional artificial viscosity term,
\begin{equation*}
	\flux[a_n, \vu_n] \coloneqq \UP[a_n, \vu_n] - h_n^\epsilon \jump{a_n}, \qquad \epsilon > -1.
\end{equation*}	
For a given $r\in Q_n$ we define
$
\tildejump{r_n} \coloneqq r_n^\text{up} - r_n^\text{down},
$
where $r_n^\text{up}$, $r_n^\text{down}$ are the upwind, downwind values of $r^n$ on the cell interface $\sigma.$
There exists $M\geq 0$ such that for all nonnegative $r\in \mathcal{Q}_n$
\begin{equation}\label{eq:avaBoundByL1Norm}
	h\sigsum \int_\sigma \ava{r} \ddd S(\vx)
	\leq
	M \int_Q r \ddd \vx.
\end{equation}

The semidiscrete VFV method \cite{feireisl2020finite, feireisl2021Book} can be formulated in the following way 
\begin{align} \label{eq:detDiscretization}
	&\int_{Q}\Dt \varrho_n \varphi \ddd \vx 
	- \sigsum\int_{\sigma} \flux[\varrho_n, \vu_n]\jump{\varphi} \ddd S= 0,
	\notag \\
    & \mbox{ for every  } \varphi \in \mathcal{Q}_n(Q);
    \notag \\[3mm]
	&\int_{Q} \Dt \vm_n \cdot  \vphi  \ddd \vx 
	-\sigsum \int_{\sigma} \flux[\vm_n, \vu_n] \cdot \jump{ \vphi} \ddd S
	\notag \\
	&
	+ \sigsum \int_{\sigma} h^{\alpha-1} {\jump{\vu_n}} \cdot \jump{ \vphi} \ddd S
	- \sigsum\int_{\sigma} \ava{p_n} \vn \cdot \jump{\vphi} \ddd S  = 0,
	\notag \\
& \mbox{ for every } \vphi \in \mathcal{D}_{n}(\vm), \quad  \vm_n^\text{in }\cdot \vn_\sigma = 0 \quad \text{ for all }\sigma \in \exFac_n;
\notag	\\[3mm]
	&\int_{Q} \Dt E_n \varphi \ddd \vx 
	-\sigsum \int_{\sigma} \flux[E_n, \vu_n]\jump{\varphi} \ddd S
	+
	h^{\alpha -1}\!\!
	\sigsum\int_\sigma \jump{\frac{\vu_n^2}2}\jump{\varphi}
	\ddd S
	\notag
	\\
	&
	-\sigsum \int_{\sigma}
	\bigg(
	\ava{p_n }\jump{\vu_n \varphi} - \ava{p_n \varphi }\jump{\vu_n}
	\bigg) 
	\cdot \vn \ddd S
	=
	0  \notag \\
   & \mbox{ for every  } \varphi \in \mathcal{Q}_n(Q).
\end{align}

We will require the following regularity of the discrete initial data.
	\begin{definition}[Admissible deterministic discrete initial data]\label{def:admissInitData}
		We say that the initial data $(\varrho_{0,n},\vm_{0,n},E_{0,n}) \in \mathcal{Q}_n([0,T]\times Q; \reals^{d+2})$ are admissible if the following holds:
		\begin{enumerate}[label=(\roman*)]
			\item $\liminf_n \essinf_{\vx\in Q}\varrho_{0,n}(\vx) > 0$ and $E_{0,n}(\vx)>\!|\vm_{0,n}(\vx)|^2/({2 \varrho_{0,n}(\vx)})$ for almost all $\vx\in Q$ and for all $n\in\naturals$,
			
			\item $\limsup_n \esssup_{x\in Q} E_{0,n}(\vx)-|\vm_{0,n}(\vx)|^2/({2 \varrho_{0,n}(\vx)}) < \infty,$
			
			\item There exists a constant $\underline{s}\in\reals$ such that $\liminf_n\essinf_{\vx} s_{0,n}(\vx)\geq \underline{s}$,

          \item
          $(\varrho_{0,n}, \vm_{0,n}, E_{0,n}) \to (\varrho_{0}, \vm_{0}, E_{0})$  in $L^\gamma(Q) \times L^{\frac{2 \gamma}{\gamma +1}}(Q; \reals^d) \times L^1(Q)$ for $n\to\infty$.
			
		\end{enumerate}
	\end{definition}

	We sum up several properties of the viscous finite volume scheme proven in Chapter 10 of Ref.~\citenum{feireisl2021Book}. For convenience, we denote the initial energy by $\mathcal{E}_{0,n} \coloneqq \int_Q E_{0,n}(\vx) \ddd \vx$.

	\begin{theorem} \label{thm:propertiesDetScheme}If the initial data $(\varrho_{0,n},\vm_{0,n},E_{0,n})$ are admissible then the VFV method (\ref{eq:detDiscretization}) has the following properties:
		\begin{enumerate}[label=(\roman*)]
			\item (Existence of approximate solutions) There exist $(\varrho_{n}, \vm_{n}, E_{n})$ satisfying (\ref{eq:detDiscretization}) for all $t\in[0,\infty)$;
			\label{thm:propDetScheme_existence}
			\item (Positivity of the density and temperature)
			\label{thm:propDetScheme_positivity} $\varrho_n,\vartheta_n>0$;
			\item (Minimal entropy principle) $s_n(t) \geq \underline{s}$;
			\label{thm:propDetScheme_minEntropy}
			\item (Stability) \label{thm:propDetScheme_stability}
			\begin{enumerate}
				\item
				$\|\varrho_n\|_{L^\infty([0,T] ;\,L^\gamma(Q))} \leq (\gamma-1)\exp\Big(-(\gamma - 1)\inf_Q s_0\Big) \mathcal{E}_{0,n}$,
				\item 
				$\|\vm_n\|_{L^\infty([0,T] ;\,L^{2\gamma/(\gamma+1)}(Q))} 
				\leq 
				\sqrt[\gamma+1]{(\gamma-1)\exp(-(\gamma-1)\underline{s})} \mathcal{E}_{0,n}$,
				\item
				$\|E_n\|_{L^\infty([0,T]; L^1(Q))} = \mathcal{E}_{0,n}$.
			\end{enumerate}
		\end{enumerate}
	\end{theorem}

	\begin{proof}
		Results \ref{thm:propDetScheme_existence}, \ref{thm:propDetScheme_positivity} are presented in Lemma 10.3 of Ref.~\citenum{feireisl2021Book}. The proof of the minimal entropy principle \ref{thm:propDetScheme_minEntropy} follows from the calculations leading to equation (10.39) of Ref.~\citenum{feireisl2021Book}. The stability estimates \ref{thm:propDetScheme_stability} follow from  (10.40), (10.41) and (10.18) of Ref.~\citenum{feireisl2021Book}, respectively.
	\end{proof}

As shown in Theorem 10.3 of Ref.~\citenum{feireisl2021Book}, the VFV method is consistent. This means that numerical solutions satisfy the generalized weak form of the Euler system up to the so-called \textit{consistency error}, which vanishes in the limit $n\to \infty$.

	\subsection{Conditional regularity of the viscous finite volume method}
	
	It is well known that for hyperbolic conservation laws equipped with a convex entropy, strong solutions exist locally for sufficiently smooth initial data. Moreover, conditional regularity results are known which guarantee global existence of strong solutions provided derivatives of the solutions remain bounded in suitable norms, see, e.g.,
	Ref.~\citenum{dafermos2005hyperbolic}. In this section, we prove an analogous result on the discrete level for the VFV method. Instead of bounds on derivatives, we will assume bounds on discrete derivatives of the form $\limsup_n |\jump{a_n}/h| <\infty$ for specific choices of $a_n$. We omit a more precise notation of jump $\jump{a_n}_\sigma$ and use $\jump{a_n}$ for simplicity.
More precisely, we will assume that the following three assumptions hold.
	
	\begin{assumption}\label{ass:det1}
		\begin{equation*}
			\limsup_n \max_{\sigma\in \inFac_{n}} |\jump{\varrho_n}/h_n| <\infty 
		\end{equation*}
	\end{assumption}

	\begin{assumption}\label{ass:det2}
		One of the following statements is true.
		\begin{enumerate}[label*=(\roman*)]
			\item $\limsup_n \max_{\sigma\in\inFac_{n} } |\jump{\vu_n}/h_n| <\infty$.
			\item There exists a constant $\underline{\varrho}>0$ such that
			$\varrho_n \geq \underline{\varrho}$ and $\limsup_n \max_{\sigma\in \inFac_{n} } |\jump{\vm_n}/h_n| <\infty$.
		\end{enumerate}
	\end{assumption}

	\begin{assumption}\label{ass:det3}
		One of the following statements is true.
		\begin{enumerate}[label*=(\roman*)]
			\item $\limsup_n \max_{\sigma\in \inFac_{n}} |\jump{\vartheta_n}/h_n| <\infty$.
			\item $\limsup_n \max_{\sigma\in \inFac_{n}} |\jump{E_n}/h_n| <\infty$.
		\end{enumerate}
	\end{assumption}

	Under the above assumptions, using suitable parameters $\alpha$, $\epsilon$, and starting from suitable initial data, we show that the VFV approximations converge strongly to a strong solution. In particular, the strong solution exists as long as the above assumptions are satisfied on the discrete level for all $n\in\naturals$.

	\begin{theorem}\label{thm:deterministic_convergence}
		Let $\{\varrho_n, \vm_n, E_n\}_{n=1}^\infty$ be a sequence of VFV approximations with $0<\alpha<4/3$ and $\epsilon>-1$ satisfying assumptions (\ref{ass:det1}) - (\ref{ass:det3}) and
		\begin{enumerate}[label=(\roman*)]
			\item $h_n\to 0$ for $n\to\infty$,
			\item the discrete initial data $(\varrho_{n,0}, \vm_{n,0}, E_{n,0})$ are admissible in the sense of Definition~2.1;
		\end{enumerate}
		Then $(\varrho_n, \vm_n, E_n) \to (\varrho, \vm, E)$  in $L^q((0,T); L^\gamma(Q) \times L^{\frac{2 \gamma}{\gamma +1}}(Q; \reals^d)\times L^1(Q))$, where $ 1 \leq q < \infty $  and  $(\varrho, \vm, E)$  is a  strong solution of (\ref{eq:EulerSystem}).
	\end{theorem}

	The rest of this section deals with the proof of  Theorem~\ref{thm:deterministic_convergence}. Therefore, from now on until the end of this section $(\varrho_n, \vm_n, E_n)$ denotes a sequence of VFV approximations with  $0<\alpha<4/3$ and $\epsilon>-1$. We assume that assumptions (\ref{ass:det1}) - (\ref{ass:det3})  and (i), (ii)  from Theorem~\ref{thm:deterministic_convergence} hold.
	
	\subsubsection{A priori bounds}

	We show suitable a priori estimates and formulate the following auxiliary result.

	\begin{lemma}\label{lem:jumpAdjacentElements}
		There exists $C_B>0$ independent of $n$ and $\vx \in Q$ such that the number of elements $K\in \mathcal{T}_n$ with $B_{h_n}(\vx)\cap K \neq \emptyset$ is bounded by $C_B$.
		Further,
		let $\vf_n\in Q_n(Q,\reals^k)$ satisfy $\max_{\sigma\in\fac_n } |\jump{\vf_n}| \leq Ch_n$ for all $n$ and let $L,K$ be elements of $\mathcal{T}_n$ with $K\cap L \neq \emptyset$. Then there exists $C^\prime>0$ independent of $n$ such that
		\begin{equation*}
			|\vf_n^{K}-\vf_n^{L}| \leq C^\prime h_n .
		\end{equation*}
	\end{lemma}
Since $Q$ is a bounded domain with Lipschitz boundary, there exists $C_Q\geq 1$ such that for each $\vx,\vy\in Q$ there is  a Lipschitz continuous curve $\gamma\colon [0,1] \to Q$ such that $\gamma(0)=\vx$, $\gamma(1)=\vy$ and $\len(\gamma)\leq C_Q \|\vx-\vy\|$. As a straightforward consequence of the previous lemma, we get the following result.

	\begin{lemma}\label{lem:boundedness_of_approximationseq}
		Let $\vf_n\in Q_n(Q,\reals^k)$ satisfy $\max_{\sigma\in\fac_n } |\jump{\vf_n}| \leq Ch_n$ and let $L,K$ be elements of $\mathcal{T}_n$. Then there exists $C''$ independent of $n$ such that
		\begin{equation*}
			|\vf_n^{K}-\vf_n^{L}| \leq C'' (\dist(K,L)+h_n).
		\end{equation*}
	\end{lemma}
	
	\begin{proof}
		For $\delta_{\text{dist}}>0$ there are $\vx \in K^\circ$ and $\vy \in L^\circ$ with $\|\vx-\vy\| \leq \dist(K,L) + \delta_{\text{dist}}$. Denote by $\gamma \colon [0,1]\to Q$ a Lipschitz continuous curve with $\gamma(0)=\vx$, $\gamma(1)=\vy$ and $\len(\gamma)\leq C_Q \|\vx-\vy\|$.
		Fixing $\delta_h \in (0,h_n)$ we choose $N_t = 2\lceil \len(\gamma([0,1]))/(h_n - \delta_h) \rceil +1 $ points $0=t_1<t_2< \dots t_{N_t}=1$ such that 
		$\len(\gamma([t_i,t_{i+1}])) = (h_n - \delta_h)/2$ for all $i=1,\dots, N_t-2$ and $\len(\gamma([t_{N_t-1},t_{N_t}])) \leq (h_n - \delta_h)/2$. Thus, $B_{h_n}(\gamma(t_{2i}))$, $i=1,\dots, (N_t-1)/2$, is an open cover of $\gamma([0,1])$.
		For each set $B_{h_n}(\gamma(t_{2i}))$ there are at most $C_B$ elements of $\mathcal{T}_n$ which cover $Q \cap B_{h_n}(\gamma(t_{2i}))$. Thus, we find a sequence of $C_1 \leq  C_B \lceil \len(\gamma([0,1]))/(h_n - \delta_h) \rceil$ elements $K_1,\dots, K_{C_1}$ such that $K_1=K$, $K_{C_1}=L$ and $K_i\cap K_{i+1}\neq \emptyset $. Applying the previous lemma with $\delta_h$ small enough such that
		$
		\lceil \len(\gamma([0,1]))/(h_n - \delta_h) \rceil \leq \len(\gamma([0,1]))/h_n + 1
		$ and taking the limit $\delta_\text{dist}\to 0$ finishes the proof. 
	\end{proof}

	\begin{remark}
		In particular, note that (\ref{ass:det2}) (ii) implies (\ref{ass:det2}) (i), since
		\begin{equation*}
			\jump{\vu_n} = \jump{\vm_n}\ava{\varrho_n^{-1}}
			-
			\xi^{-2}\jump{\varrho_n}\ava{m_n}
		\end{equation*} 
		with $\xi\in\conv\{\varrho_n^\mathrm{in},\,\varrho_n^\mathrm{out}\}$. The next lemma shows that the reverse implication is true as well.
	\end{remark}

	\begin{lemma}\label{lem:density_lowerBound}
		If $\max_{\sigma\in \fac_n} |\jump{\vu_n}| \leq C h_n$ for all $n\in\naturals$, then there exists $\underline{\varrho}$ independent of $n$ such that $\varrho_n\geq \underline{\varrho}$ for $n$ large enough.
	\end{lemma}

	To prove the above lemma, we will use the following auxiliary lemma that is a slight modification of Lemma 10.2 of Ref.~\citenum{feireisl2021Book}.
	
	\begin{lemma}\label{lem:superTransport}
		Let $r_n \colon [0,T] \to \mathcal{Q}_n$ satisfy $r_n(0)\geq 0$, the bound $\sup_{t,\vx} |r_n(t,\vx)|<\infty$ and for all $\varphi_n\in\mathcal{Q}_n, \varphi_n\geq 0$ holds
		\begin{equation}\label{eq:subTransport}
			\frac{d}{d\,t} \int_Q r_n \phi_n \ddd \vx
			-
			\sigsum \int_\sigma \flux\big[r_n,\, \vu_n\big] \jump{\varphi_n} \ddd S(\vx) \geq 0 .
		\end{equation}
		Then $r_n(t)\geq 0$ for all $t\in[0,T]$.
	\end{lemma}

	\begin{proof}[Poof of Lemma~\ref{lem:density_lowerBound}]
	As we assume in (ii) that $\liminf_n\varrho_{0,n} \geq \underline{\varrho}_0 > 0$, there exists $\underline{\varrho}(0)\in \reals_{>0}$ such that $\varrho_n(0,\vx) \geq  \underline{\varrho}(0)$ for all $\vx\in\Omega$ and $n$ sufficiently large. We define for $t>0$
	\begin{equation*}
		\underline{\varrho}(t) \coloneqq \underline{\varrho}(0) \exp\left( -Lt \right)
	\end{equation*}
	with $L\geq C\,M$. Using (\ref{eq:avaBoundByL1Norm}), we derive the following inequality for arbitrary nonnegative $\varphi_n\in \Qh(\Omega)$ 
	\begin{align*}
		&\frac{d}{dt} \int_\Omega \underline{\varrho} \varphi_{n} \ddd \vx
		-
		\sigsum \int_\sigma F_h^\text{up}[\underline{\varrho}, \vu_n]
		\jump{\varphi_n} \ddd S(\vx)
		\\
		&
		=
		- L \underline{\varrho} \int_\Omega \varphi_n \ddd \vx
		-
		\underline{\varrho}
		\sigsum \int_\sigma 
		\ava{\vu_n} \cdot \vn \jump{\varphi_n} 
		\ddd S(\vx)
		\\
		&
		=
		- L \underline{\varrho} \int_\Omega \varphi_n \ddd \vx
		+
		\underline{\varrho}
		\sigsum \int_\sigma 
		\jump{\vu_n} \cdot \vn \ava{\varphi_n} 
		\ddd S(\vx)
		\\
		&
		\leq
		\underline{\varrho} \int_\Omega \varphi_n \ddd \vx \left( -L + C M \right)
		\leq 0.
	\end{align*}
	Choosing $r_n = \varrho_n-\underline{\varrho}$ in Lemma~\ref{lem:superTransport}, we derive that $\varrho_n(t) \geq \underline{\varrho} \equiv \underline{\varrho}(T) >0$ for all $t\in [0,T]$. 
\end{proof}
Consequently, we get a uniform lower bound on $\vartheta_n$ for $n$ large enough due to Theorem~\ref{thm:propertiesDetScheme}~\ref{thm:propDetScheme_minEntropy}. Additionally, the bound of $\varrho \vu^2$, cf.~Theorem~\ref{thm:propertiesDetScheme}(iv)(c) and assumption \eqref{ass:det2}, yield uniform boundedness of $u_n(t)$ in $L^\infty(Q)$. 

With very similar arguments, it is straightforward to show an upper bound for the internal energy.

\begin{lemma}\label{lem:temp_upper_bound}
	Under the assumptions of Theorem~\ref{thm:deterministic_convergence} there exists $\overline{\varrho e}$ such that $\varrho_n e(\varrho_n, \vm_n, E_n) \leq\overline{\varrho e}$ for almost all $(t,\vx)$ and $n$ sufficiently large.
\end{lemma}

In particular, Theorem~\ref{thm:propertiesDetScheme}~\ref{thm:propDetScheme_minEntropy} implies that there also exist a uniform upper bound on $\varrho_n$ for $n$ sufficiently large.

\subsubsection{Lipschitz continuous approximation of numerical solutions}

Next, we show that there exists a sequence of uniformly Lipschitz continuous approximations of the numerical solutions. To this end, we first note that the cell center values $(\varrho_n^K(t), \vm_n^K(t), E_n^K(t))$, $K\in\mathcal{T}_n$, are uniformly Lipschitz continuous in time with respect to both $K$ and $n$.

\begin{lemma}[Lipschitz continuity in time]\label{lem:LipschitzInTime}
	If assumptions \ref{ass:det1}, \ref{ass:det2} and \ref{ass:det3} are satisfied then $t\mapsto \varrho_n^K(t)$, $t\mapsto \vm_n^K(t)$ and $t\mapsto E_n^K(t)$ are globally Lipschitz continuous 
	 uniformly in $K\in \mathcal{T}_n$ and $n\in\naturals$.
\end{lemma}

\begin{proof}
	By definition, $\{\varrho_n^K, \vm_n^K, E_n^K\}_{K\in\mathcal{T}_n}$ form the solution of the system of ordinary differential equations (\ref{eq:detDiscretization}). 
	To prove uniform Lipschitz continuity, it is sufficient to bound the right hand side of the respective ordinary differential equation for each cell average  $\{\varrho_n^K, \vm_n^K, E_n^K\}_{K\in\mathcal{T}_n}$ in $L^\infty([0,T])$ for all $K\in\mathcal{T}_n$.

 	As the jumps of all quantities are bounded by $C\:h_n$ for some positive constant $C>0$ independent of $n$ and $K$, all terms containing a jump are bounded. For the other terms we note that
 	\begin{align*}
 		&\sum_{\sigma \in \fac(K)} \int_\sigma \ava{a_n} \ava{\vb_n} \cdot \vn_K \ddd S(\vx)
 		\\
 		&= 
 		\sum_{\sigma \in \fac(K)} \int_\sigma \left(a_n^\text{in} + \frac12 \jump{a_n}\right) \left(\vb_n^\text{in} + \frac12 \jump{b_n} \right) \cdot \vn_K \ddd S(\vx)
 		\\
 		&=  
 		\sum_{\sigma \in \fac(K)} \int_\sigma a_n^\text{in} \vb_n^\text{in}\cdot \vn_K \ddd S(\vx)
 		+ \mathcal{O}(h_n^d) = \mathcal{O}(h_n^d),
 	\end{align*}
 if both $\max_{\sigma\in \fac_n} |\jump{a_n}| = \mathcal{O}(h_n)$ and $\max_{\sigma\in \fac_n} |\jump{b_n}| = \mathcal{O}(h_n)$ for $a_n \in \Qh(Q)$ and $\vb_n \in \Qh(Q;\reals^d)$ with $\|a_n\|_{L^\infty}$ and $\|\vb_n\|_{L^\infty}$ bounded uniformly in $n$. And similarly,
 \begin{equation*}
 	\sum_{\sigma \in \fac(K)} \int_\sigma \ava{a_n} \vb \cdot \vn_K \ddd S(\vx)
 	 = 
 	 \mathcal{O}(h_n^d) 
 	 =
 	 \!\!\!
 	\sum_{\sigma \in \fac(K)} \int_\sigma  a \ava{\vb_n} \cdot \vn_K \ddd S(\vx)
\end{equation*}
 for $a\in \reals$ and $\vb\in\reals^d$. Therefore $\int_K \frac{d}{dt} f_n \ddd x = 
 \mathcal{O}(h_n^d) $ for $f\in \{\varrho,\vm, E\}$.
 
\end{proof}

Next, for fixed $t\in [0,T]$, we approximate the numerical solutions by uniformly Lipschitz continuous functions in space such that the approximation error vanishes in $L^\infty(Q)$ for $n\to\infty$.

\begin{lemma}\label{lem:LipApproximations}
	Let $\{\vf_n\}_{n\in\naturals}$ denote a sequence of functions $\vf_n\in \mathcal{Q}_n(Q,\reals^k)$ with $\max_{\sigma\in\fac_n } |\jump{\vf_n}| \leq Ch_n$ for some $C>0$. Then there exists a sequence $\hat \vf_n$ of uniformly Lipschitz continuous functions $\hat \vf_n \in C^{0,1}(Q,\reals^k)$ satisfying
	$\|\vf_n-\hat \vf_n\|_{L^\infty} \to 0$ for $n\to\infty$.
\end{lemma}

\begin{proof}
	We may assume $k=1$ without loss of generality. For {\cb a node of the mesh $\vx$} we define $\hat f_n(\vx)$ as the average of $f_n^K$ for all $K$ with $\vx\in K$, i.e.
	\begin{equation*}
	\hat f_n(\vx) = (\#\{K\in \mathcal{T}_n\, \vx\in K\})^{-1}\sum_{K\in \mathcal{T}_n,\, \vx\in K} f_n^K.
	\end{equation*}
	Note that $|\vf_n^{K}-\vf_n^{L}| \leq {\cb C'} h_n$ for all $K,L$ with $\vx\in K \cap L$ due to Lemma~\ref{lem:jumpAdjacentElements}. In particular,
	\begin{equation*}
			|\hat \vf_n(\vx)-\vf_n^{K}| \leq {\cb C'} h_n \qquad \text{for all }K\text{ with }\vx\in K.
	\end{equation*}
	Fixing $K=\conv\{\vx_0,\dots, \vx_d\} \in\mathcal{T}_n$, we define $\hat f_n(\vx)$ for $\vx = \sum_i \lambda_i \vx_i \in K$ by $\hat f_n(\sum_i \lambda_i \vx_i) = \sum_i \lambda_i f_n(\vx_i)$. We may assume $\vx_0=\va_K$ and $\vx_i=h_n \matA_K \ve{e}_i+\va_K$ for $i=1,\dots,d$. 
	In particular,
	\begin{equation*}
		\hat f_n (\vx)|_K = h_n^{-1}(\hat f_n(\vx_1)-\hat f_n(\vx_0),\dots, \hat f_n(\vx_d)-\hat f_n(\vx_0))A_K^{-1} (\vx-\vx_0) + f_n(\vx_0).
	\end{equation*}
Thus, for each $\vx \in Q$ there exists a neighbourhood $U$ such that $\hat f_n|_U$ is Lipschitz continuous with Lipschitz constant less than or equal to $2 d {\cb C'}/c_T$. 
 Here, $c_T^2 \leq \lambda$ for all eigenvalues $\lambda$ of $\matA_K^T \matA_K$ for all $K\in\mathcal{T}_n$, $n\in\naturals$.
In particular, $\hat f_n$ is Lipschitz continuous on $\overline{Q}$ with Lipschitz constant less than or equal to $2 d {\cb C'} C_Q/c_T$, independent of $n$.
\end{proof}

In summary, the above two Lemmas imply that on $[0,T]\times \overline{Q}$ there exist a sequence of uniform globally Lipschitz continuous approximations of the VFV solutions such that the approximation error vanishes as $n \to \infty$.

\subsubsection{Proof of Theorem~\ref{thm:deterministic_convergence}}
We are now ready to prove  Theorem~\ref{thm:deterministic_convergence}.

\begin{proof}[Proof of Theorem~\ref{thm:deterministic_convergence}]
	Due to Lemmas~\ref{lem:density_lowerBound} and \ref{lem:temp_upper_bound}, $\varrho_n$ is bounded away from $0$ and $\vartheta_n$ is bounded from above. Moreover, according to Lemmas
	\ref{lem:LipschitzInTime} and \ref{lem:LipApproximations} there exits a sequence of uniformly Lipschitz continuous functions $\{\hat \varrho_n,\, \hat\vm_n,\,\hat{E}_n\}_{n=1}^\infty$ with
	\begin{equation*}
		\varrho_n-\hat{\varrho}_n\to 0,\qquad
		\vm_n-\hat{\vm}_n\to 0, \qquad 
		E_n-\hat{E}_n\to 0 \qquad \text{uniformly in }[0,T]\times \overline{Q}
	\end{equation*}
	for $n\to \infty$. Further, Lemma~\ref{lem:boundedness_of_approximationseq} implies the boundedness of $(\hat \varrho_n,\, \hat\vm_n,\,\hat{E}_n)$. Thus, by the Arzelà-Ascoli theorem,
	after possibly passing to a subsequence and relabeling the sequence, $(\hat \varrho_n,\, \hat\vm_n,\,\hat{E}_n)$ converges uniformly to the Lipschitz function $(\varrho, \vm, E)\in \Lip([0,T]\times \overline{Q};\reals^{d+2})$. Note that this convergence is pointwise a.e. and thus, the nonlinear terms in the momentum and energy equations converge as well. We will show that $(\varrho,\vm,E)$ is a strong solution of (\ref{eq:EulerSystem}) proving that passing to subsequences is not necessary. Choosing $\chi(x) = \min\{x, a\}$ for $a\in \reals$ large enough,
	Theorem 10.3 of Ref.~\citenum{feireisl2021Book} implies that
	 $(\varrho,\, \vm,\, E)$ is a Lipschitz continuous weak solution of the Euler system with density bounded away from $0$ and initial data $(\varrho_0,\, \vm_0,\, E_0)$. Lemma ~\ref{lem:uniquenessStrongSolutions} implies that the limit is unique and there is no need to take subsequences.

\end{proof}

 \section{The stochastic collocation viscous finite volume scheme}
 
Let  $[\Omega, \mathcal{B}(\Omega), \mathds{P}]$ be
a complete probability space with $\Omega$ a compact metric space of samples, $\mathcal{B}(\Omega)$  $\sigma$-algebra of Borel subsets of $\Omega$, and $\mathds{P}$  a complete Borel probability measure on $\Omega$. Let us denote by $D$ the set of data 
$D\equiv \{ (\varrho, \vm, E) \in  L^\gamma(Q;\reals) \times L^{\frac{2 \gamma}{\gamma +1}}(Q; \reals^d) \times L^1(Q;\reals)|\, {\rm{inf}}_{x \in Q} \varrho > 0, {\rm{inf}}_{x \in Q} \ E - \frac{1}{2} \frac{| \vm|^2 }{\varrho} > 0 \}. $
We assume that the randomness is enforced through the random initial data. This means  that 
the mapping
\begin{equation*}
 	\left( \varrho_0, \vm_0, E_0 \right)(\omega) \colon \Omega \longrightarrow D
 \end{equation*}
 is Borel measurable for a.a. $\omega \in \Omega.$

To approximate the random Euler system, we divide $\Omega$ into a sequence of finite partitions with shrinking maximal diameter. Having a piecewise constant deterministic approximation on $Q$, we also apply a piecewise constant approximation in the random space on $\Omega$. Other choices are possible, and we refer to Ref.~\citenum{CHKL}, where a higher order WENO approximation on $\Omega$ was successfully applied for hyperbolic conservation laws.
 The approximate random solutions are defined in the following way. 
 \begin{definition}[Approximate random solution]
 	Given a partition $\{\Omega_m^M\}_{m=1}^{\nu(M)}$ of $\Omega$ and a set of collocation nodes $\omega_m^M\in \Omega_m^M$ we define the approximate random solution of the random Euler equations as
 	\begin{align}
 		\varrho_n^M(t,\vx,\omega)
 		& =
 		\sum_{m=1}^{\nu(M)} \ind{\Omega_m}(\omega) \varrho_n(t,\vx,\omega_m)
 		\notag \\
 		\vm_n^M(t,\vx,\omega)
 		& =
 		\sum_{m=1}^{\nu(M)} \ind{\Omega_m}(\omega) \vm_n(t,\vx,\omega_m)
 		 \notag \\
 		E_n^M(t,\vx,\omega)
 		& =
 		\sum_{m=1}^{\nu(M)} \ind{\Omega_m}(\omega) E_n(t,\vx,\omega_m) \quad \mbox{ for } a.a.\  \omega \in \Omega,
 \label{randsol}
 	\end{align}
 	where $(\varrho_n(t,\vx,\omega_m), \vm_n(t,\vx,\omega_m),  E_n(t,\vx,\omega_m))$ denote the (deterministic) solution of the VFV method corresponding to the mesh $\mathcal{T}_n$ evaluated at $(t,x)$ with the initial data $(\varrho_0(\cdot,\omega_m), \vm_0(\cdot,\omega_m),  E_0(\cdot,\omega_m))$.
 \end{definition}
 
 To guarantee the convergence of $(\varrho_n^M(0,\cdot,\cdot), \vm_n^M(0,\cdot,\cdot),  E_n^M(0,\cdot,\cdot))$ in the limit as $M \to \infty ,n\to\infty$ we require that the initial data are admissible.
  
  \begin{definition}[Admissible random initial data]
 	The initial data $(\varrho_0,\, \vm_0,\, E_0) \colon \Omega \to D $
 are admissible, if 
 	\begin{enumerate}[label=(\roman*)]
 		\item 	there exist a.e.\ continuous functions $\underline{\varrho}\colon \Omega \to \reals_{>0}$ and $\cb \overline{E}\colon \Omega \to \reals_{>0}$ and $\underline{s}\colon \Omega \to \reals$ such that for all $\omega\in \Omega$ and $\vx\in Q$
 		\begin{equation*}
 			{
 				\arraycolsep=0.5mm
 				\begin{array}{rlr}
 					\displaystyle \essinf_{\vy}\varrho_0(\vy, \omega) &>\underline{\varrho}(\omega),\qquad
 					&
 					\displaystyle \esssup_{\vy}E_0(\vy, \omega) <\overline{E}(\omega),
 					\\[5mm]
 					\displaystyle e_0(\vx, \omega) & > 0,\qquad
 					&
 					\displaystyle \essinf_{\vy} s_0(\vy, \omega) > \underline{s}(\omega);
 				\end{array}
 			}
 		\end{equation*}
 	
 	\item
  	there exists a.e.\ continuous functions $r_\varrho,\, r_\vm,\, r_S \colon \Omega \to \reals_{\geq0}$ such that
 	\begin{equation*}
 		\|\varrho_0(\cdot,\omega)\|_{L^\gamma} \leq r_\varrho\,, \quad
 		\|\vm_0(\cdot,\omega)\|_{L^{2\gamma/(\gamma+1)}} \leq r_\varrho\,, \quad 
 		\|S_0(\cdot,\omega)\|_{L^\gamma} \leq r_S\,,
 	\end{equation*}
 where $S_0 = \varrho_0 s_0,$ $s_0 = s(\varrho_0, \vm_0, E_0)$;	
 	
 	\item the following maps from $\Omega$ to $\reals$  a.e.\ continuous:
 	\begin{align}\label{eq:RiemannIntegrabilityOfInitialData}
 		F_{\varrho,\varphi} \colon \omega\in \Omega &\mapsto \int_Q \varrho_0(\vx, \omega) \varphi(\vx) \ddd \vx, \qquad \varphi \in C_c^\infty(Q),
 		\notag
 		\\
 		F_{\vm,\varphi} \colon\omega\in \Omega &\mapsto \int_Q \vm_0(\vx, \omega) \cdot \vphi(\vx) \ddd \vx, \qquad \vphi \in C_c^\infty(Q;\reals^d),
 		\notag
 		\\
 		F_{S,\varphi} \colon\omega\in \Omega &\mapsto \int_Q S_0(\vx, \omega) \varphi(\vx) \ddd \vx, \qquad \varphi \in C_c^\infty(Q),
 		\notag
 		\\
 		F_{E} \colon\omega\in \Omega &\mapsto \int_Q E_0(\vx, \omega) \ddd \vx.
 	\end{align}
 	\end{enumerate}

 \end{definition}

We present the following result, which strengthens the convergence of the piecewise constant discretizations of the initial data.
\begin{lemma}\label{lem:ConvergenceOfInitialData}
	Let $(\varrho_0,\, \vm_0,\, E_0)$ be admissible random initial data. Then 
	\begin{align}\label{eq:initialDataLemma}
		\sum_{m=1}^{\nu(M)} \ind{\Omega_m}(\omega)\varrho_0(\cdot,\omega_m) &\to 	\varrho_0(\omega) 
		\text{ in }L^\gamma(Q; \reals), 
		\notag
		\\
		\sum_{m=1}^{\nu(M)} \ind{\Omega_m}(\omega)\vm_0(\cdot,\omega_m) &\to \vm_0(\omega)
		\text{ in }L^{2\gamma/(\gamma+1)}(Q;\reals^d),
		\notag
		\\
		\sum_{m=1}^{\nu(M)} \ind{\Omega_m}(\omega)E_0(\cdot,\omega_m) &\to E_0(\omega) 
		\text{ in }L^1(Q; \reals) \quad \mbox{ as } M \to \infty, \mathds{P}-a.s.
	\end{align}
\end{lemma}

\begin{proof}
	Let $\{\varphi_k\}_{k=1}^\infty\subset C^\infty_c(Q)$ and $\{\vphi_k\}_{k=1}^\infty\subset C^\infty_c(Q;\reals^d)$ denote dense sequences in the respective space of smooth functions. As the initial data are admissible, there exists a set $\mathcal{N} \subset \Omega$ with $\mathds{P}(\mathcal{N}) = 0$ such that the maps in (\ref{eq:RiemannIntegrabilityOfInitialData}) are continuous in each point of the complement $\mathcal{N}^c$ of $\mathcal{N}$ for test functions $\varphi_k$ and $\vphi_k$ .
	Clearly, in every point of continuity $\omega$ of $F\colon \Omega\to \reals$,
	\begin{equation*}
		\sum_{m=1}^{\nu(M)} \ind{\Omega_m^M}(\omega) F(\omega_m^M) \to F(\omega).
	\end{equation*}
	This implies that for $M\to \infty $ we have
	\begin{align*}
		\int_Q \sum_{m=1}^{\nu(M)} \ind{\Omega_m}(\omega)\varrho_0(\vx,\omega_m) \varphi_k(\vx) \ddd \vx &\to \int_Q	\varrho_0(\omega) \varphi_k(\vx) \ddd \vx
		, 
		\notag
		\\
		\int_Q \sum_{m=1}^{\nu(M)} \ind{\Omega_m}(\omega)\vm_0(\vx,\omega_m) \cdot \vphi_k(\vx) \ddd \vx &\to \int_Q  \vm_0(\omega)  \cdot \vphi_k(\vx) \ddd \vx
		,
		\notag
		\\
		\int_Q \sum_{m=1}^{\nu(M)} \ind{\Omega_m}(\omega)S_0(\vx,\omega_m) \varphi_k(\vx) \ddd \vx &\to \int_Q  S_0(\omega) \varphi_k(\vx) \ddd \vx
		,
		\\
		\int_Q \sum_{m=1}^{\nu(M)} \ind{\Omega_m}(\omega)E_0(\vx,\omega_m) \ddd \vx &\to \int_Q E_0(\vx, \omega)  \ddd \vx
	\end{align*}
for all $\omega \in \mathcal{N}^c$. 

Now, let us use the following notation 
\begin{equation}
\label{tildeeq}
{f}_0^M(\omega) = \sum_{m=1}^{\nu(M)} \ind{\Omega_m}(\omega) f_0(\cdot,\omega_m)
\end{equation}
for any 
$f_0\in\{ \varrho_0, \vm_0, E_0, S_0 \}$. 
The uniform bounds of  ${\{{\varrho}_0^M, {\vm}_0^M, {S}_0^M \}}_{M=1}^\infty$ in $L^\gamma(Q)\times L^{2\gamma/(\gamma+1)}(Q;\reals^d)\times L^\gamma(Q)$ imply weak convergence of $({\varrho}_0^M,  {\vm}_0^M,  {S}_0^M)$ as $M\to\infty$. Following similar arguments as in 
Section~6 of Ref.~\citenum{Chaudhuri2021} we apply the convexity arguments for $E(\varrho_0, \vm_0, S_0)$ and  obtain  the strong convergence claimed in  (\ref{eq:initialDataLemma}). 
\end{proof}

Lemma~\ref{lem:ConvergenceOfInitialData} yields $\mathds{P}$-a.s.\ convergence
of discrete admissible random initial data $(\varrho_0^M,  \vm_0^M,  S_0^M)$ as $M\to \infty$. To discretize the initial data on $Q$, we apply the projection operator 
\begin{equation*}
	\Pi_n  \colon L^p(Q) \to L^p(Q), \qquad f \mapsto \sum_{K\in\mathcal{T}_n}  \frac{\ind{K}}{|K|} \int_K f(\vy)\ddd \vy.
\end{equation*}
Since the projection error is controlled in the $L^p$-norm, we also get a.e.\ convergence in $Q$ of fully discretized initial data. Choosing the space discretization parameter $n=n(M)$, such that $n \to \infty$ as $M \to \infty$, we now denote for any $f_0\in\{ \varrho_0, \vm_0, E_0,  S_0 \}$ its piecewise constant projection by
\begin{equation*}
	f_{0,n}^M(\omega) = \sum_{m=1}^{\nu(M)} \ind{\Omega_m}(\omega) \Pi_n f_0(\cdot,\omega_m).
\end{equation*}

\begin{lemma}\label{lem:fullyDiscreteInitialDataConvergence}
		Let $(\varrho_0,\, \vm_0,\, E_0)$ be admissible random initial data. Then there exists a Borel measurable map {\cb
	\begin{equation*}
		\omega\in \Omega \mapsto (\varrho_0^*(\omega),\, \vm^*_0(\omega),\, E_0^*(\omega)) \in D
	\end{equation*} 
	such that
	\begin{equation*}
		\varrho_0 = \varrho^*_0, \quad \vm_0 = \vm^*_0, \quad E_0 = E^*_0 \qquad \mathds{P}-\mbox{a.s.}
	\end{equation*}
	and
	\begin{align}\label{eq:aeConvergenceOfFullyDiscreteInitialData}
		&\varrho_{0,n}^M(\omega) \to  \varrho^*_0(\omega) \text{ in } L^\gamma(Q),
		\quad 
		\vm_{0,n}^M(\omega) \to  \vm^*_0(\omega) \text{ in } L^{2\gamma/(\gamma+1)}(Q;\reals^d),
		\notag
		\\&
		\quad
		E_{0,n}^M(\omega) \to E^*_0(\omega) \text{ in } L^1(Q) \quad \mbox{ as } M, n=n(M) \to \infty  \qquad  \mathds{P}-\mbox{a.s.}
	\end{align}
}
	\noindent Further, for all ${ n\in\naturals}$ $M\in\naturals$, $m\leq \nu(M)$ and all $\omega\in \Omega_m^M$ 
	\begin{align}\label{eq:discreteInitialDataProperties}
		&\essinf_{\vx\in Q} \varrho^M_{0, n}(\vx,\omega){ \geq\underline{\varrho}(\omega_m^M)}>0,\qquad
		&&{\esssup_{\vy} E^M_{0, n}(\vy, \omega) \leq \overline{E}({\cb \omega_m^M})},
		\notag
		\\
		&
		\essinf_{\vx\in Q} e^M_{0, n}(\vx,\omega)>0,
		\qquad
		&&\essinf_{\vy} s^M_{0, n}(\vy, \omega) > \underline{s}(\omega_m^M).
	\end{align}
\end{lemma}

\begin{proof}
	{\cb
		As the pointwise limit of a Borel measurable map ranging in a metric space is again Borel measurable, it suffices to prove (\ref{eq:aeConvergenceOfFullyDiscreteInitialData}) with $\varrho_0$, $\vm_0$, $E_0$ instead of $\varrho^*_0$, $ \vm^*_0$, $E^*_0$. This will then immediately imply the existence of $\varrho^*_0$, $\vm^*_0$, $E^*_0$. }

To show  (\ref{eq:aeConvergenceOfFullyDiscreteInitialData}) let us consider  $f_0 \in \{\varrho_0,\, \vm_0,\,E_0\}$.	
	We fix an arbitrary positive $\varepsilon>0$ and choose $g\in C^\infty_c(Q)$ with $\|f_0 - g\|_{L^p} \leq \varepsilon$ for the corresponding $p \in \{\gamma, 2 \gamma / (\gamma + 1),1\}$. 
Thus,	using the notation from \eqref{tildeeq}  and Jensen's inequality we derive 
\begin{align*}
		&\|  f_{0,n}^M - f_0\|_{L^p} \leq 
		\| \Pi_n [{f}_0^M - f_0] \|_{L^p(Q)}
		+
		\| \Pi_n [f_0 - g] \|_{L^p(Q)}
		\\ 
&+	\| \Pi_n [g] - g \|_{L^p(Q)}
		+
		\| g -f_0 \|_{L^p(Q)}
		\\
		&\leq 
		\|{f}_0^M - f_0\|_{L^p} + 2 \| g -  {f}_0 \|_{L^p} + \| \Pi_n [g] - g \|_{L^p}
	\end{align*}
with the last term vanishing for $n(M)\to\infty$ due to the smoothness of $g$.  

For every point of continuity $\omega$ of $\underline{\varrho}$, $\overline{e}$ and $\underline{s}$, the properties listed in (\ref{eq:discreteInitialDataProperties}) follow from the respective properties of admissible random initial data.
The positivity of $\varrho_{0,n}^M$ and $\vartheta_{0,n}^M$ follows immediately from the properties of admissible random initial data. Likewise, for every point of continuity $\omega$ of $\underline{\varrho}$, $\overline{\varrho}$ and $\underline{s}$ the properties listed in (\ref{eq:discreteInitialDataProperties}) follow from the respective properties of admissible random initial data. Here, the lower bound on the density and upper bound on the energy  are obvious. For the positivity of the initial internal energy and for the lower bound on the initial entropy, we use Jensen's inequality. More precisely, as $(\varrho,\vm) \mapsto \vm^2/\varrho$ is convex on $\reals^d \times (0,\infty)$, we find that 
\begin{align*}
	&\varrho_{0,n}^M (\omega,\cdot) e_{0,n}^M(\omega,\cdot) 
	\\
	&= \sum_{m=1}^{\nu(M)} \ind{\Omega_m}(\omega) \bigg( \Pi_n[E_0(\omega_m)] - \Pi_n[\vm_0(\omega_m)]^2/\Pi_n[\varrho_0(\omega_m)] \bigg)
	\\&
	\geq
	\sum_{m=1}^{\nu(M)} \ind{\Omega_m}(\omega) \bigg( \Pi_n[E_0(\omega_m)] - \Pi_n\Big[\vm_0(\omega_m)^2/\varrho_0(\omega_m)\Big] \bigg) >0.
\end{align*}
Similarly, as $\varrho \mapsto \varrho^{\gamma}$ is convex and $(\varrho, \vm, E) \mapsto E - \vm^2/\varrho$ is concave,
\begin{align*}
	&\Pi_n\Big[\varrho_0(\omega_m)\Big]^\gamma 
	\leq 
	\Pi_n\Big[\varrho_0(\omega_m)^\gamma\Big]
	\\
	&
	\leq 
	\Pi_n\Big[(\gamma-1)\exp(-\underline{s}/c_v)(E_0(\omega_m)-\vm_0(\omega_m)^2/\varrho_0(\omega_m))\Big]
	\\
	&
	\leq
	(\gamma-1)\exp(-\underline{s}/c_v)\Big(\Pi_n[E_0(\omega_m)] - \Pi_n[\vm_0(\omega_m)]^2/\Pi_n[\varrho_0(\omega_m)] \Big).
\end{align*}
Thus
\begin{equation*}
	s_{0,n}^M(\omega, \cdot) = c_v \log\left( (\gamma -1)\frac {\varrho_{0,n}^M(\omega,\cdot) e_{0,n}^M(\omega,\cdot)}{\varrho_{0,n}^M(\omega,\cdot)^{\gamma }} \right) \geq \underline{s}(\omega_m)
\end{equation*}
for $m$ with $\omega\in \Omega_m^M$. Using the a.e.\ continuity of $\underline{\varrho}$, $\overline{\varrho}$ and $\underline{s}$ in $\omega$ we conclude the statement of the lemma.
\end{proof}

 \subsection{Main result: convergence}
 
This section is devoted to our main result:  convergence of the stochastic collocation VFV method applied to the random Euler equations.
 Our working hypothesis will be that Assumptions~\ref{ass:det1} - \ref{ass:det3} hold in probability.
 
 \begin{theorem}[Convergence in probability]\label{thm:conv_in_probability}
 	Let the initial data  \linebreak $(\varrho_0, \vm_0, E_0)$ be admissible random data. Let the family of sets
 	$(\Omega_m^M)_{m=1,  M \in \naturals}^{\nu(M)}$ satisfy \  $\Omega_m^M\cap \Omega_\ell^M = \emptyset \ \ \forall \ell \neq m$,
 	\begin{equation*}
 		\max_{m=1,\dots, \nu(M)}\diam (\Omega_m^M) \to 0 \mbox{ as } M \to \infty,
 		\quad 
 		\bigcup_{m=1}^{\nu(M)} \Omega_m^M = \Omega.
 	\end{equation*}
 	Let $h_n >0$  and  $h_n \to 0$ for $n=n(M) \to \infty$. Finally, let the solutions ${\{\varrho_{n(M)}^{M},\vm_{n(M)}^{M},E_{n(M)}^{M} \}}_{M=1}^\infty$ obtained by the stochastic collocation VFV method \eqref{randsol} have bounded discrete gradients in probability. More precisely, this means that for each $\varepsilon>0$ there exists $N=N(\varepsilon)$ such that
 	\begin{equation}\label{eq:stcoasticJumpAssumption}
 		\limsup\limits_{M,n(M)\to\infty}\mathds{P}\left(\max_{\sigma\in \inFac}\max\left\{
 		\Big|\jump{\varrho_{n}^{M}}\Big|,  \|\jump{\vu_{n}^{M}}\|, 
 		\Big|\jump{E_{n}^{M}}\Big| \right\}/h_n 
 		\geq N\right) \leq \varepsilon.
 	\end{equation}
 	Then for  $ M \to \infty, n(M) \to \infty $
 	\begin{eqnarray*} 		
 &&(\varrho_{n(M)}^M, \vm_{n(M)}^M, E_{n(M)}^M) \to (\varrho, \vm, E) \text{ strongly in } \\ 
 && L^q((0,T); L^\gamma (Q) \times L^{\frac{2 \gamma}{\gamma +1}} (Q; \reals^d) \times L^1(Q)), 
 		\text{ in probability,} \ 1 \leq q < \infty.	
  	\end{eqnarray*}
 \end{theorem}
 \begin{proof}
 	For $K>0$ large enough, there is a compact embedding from $W^{K,2}_0$ to $C(\overline{(0,T)\times Q})$, see e.g.\ Theorem {\cb 6.3} of Ref.~\citenum{adams2003sobolev}. Schauder's theorem implies that the adjoint of the above embedding is compact as well. Since
 	Theorem~\ref{thm:propertiesDetScheme}~\ref{thm:propDetScheme_stability} implies boundedness in probability of the sequence of approximate solutions in $L^q((0,T);  L^\gamma(Q) \times L^{2 \gamma / (\gamma +1)}(Q; \reals^d) \times L^1(Q))$,
 	the sequence is tight in $W^{-K,2}((0,T)\times Q; \reals^{d+2})$ for  $K$ enough large. 
 	 
 	Further, defining
 	\begin{equation*}
 		C^{\cb M}(\omega) = \max_{\sigma\in \inFac,\, \cb t}\max\left\{
 		\Big|\,\jump{\varrho_{n(M)}^M}\,\Big|, \;  \left\|\,\jump{\vu_{n(M)}^M}\,\right\|, \;
 		\Big|\,\jump{E_{n(M)}^M}\,\Big| \right\}/h_{n\cb(M)},
 	\end{equation*}
 assumption (\ref{eq:stcoasticJumpAssumption}) implies tightness of $C^{\cb M}$ in $\reals$.
 	Finally, by Lemma~\ref{lem:ConvergenceOfInitialData}, assumptions on admissible  initial data
 	imply convergence of $(\varrho_{0,n}^M, \vm_{0,n}^M, E_{0,n}^M)$ in $L^\gamma(Q) \times L^{2\gamma/(\gamma + 1)}(Q)\times L^1(Q)$ as $M,n(M) \to \infty$ $\mathds{P}$-a.s.
Consequently, the initial data also converge in law and the application of the Prokhorov theorem, Theorem 5.2 of Ref.~\citenum{billingsley1999}, yields tightness of the initial data.

Let us  define the following separable Banach space (i.e.~in particular a Polish space)
 	\begin{align*}
 		X = 
 		&\left( L^\gamma(Q) \right)^3 
 		\times \left(L^{2\gamma/(\gamma + 1)}(Q;\reals^d) \right)^3
 		\times  \left( L^1(Q) \right)^3 
 		\times \left( W^{-K,2}((0,T)\times Q) \right)^2 \\
 		&
 		\times \left( W^{-K,2}((0,T)\times Q;\reals^d) \right)^2
 		\times \left( W^{-K,2}((0,T)\times Q) \right)^2
 		\times \reals^{\cb 2}.
 	\end{align*}

 Given two arbitrary but fixed subsequences $\{M_{1,k}\}_{k\in\naturals}$ and $\{M_{2,k}\}_{k\in\naturals}$ of $\{M\}_{M\in\naturals}$, we define the sequence $\{U_k\}_{k\in\naturals}$ ranging in $X$ by
 	\begin{align*}
 		U_k = \Big(
 		&\varrho_0, \, \varrho_{0}^{M_{1,k}}, \, \varrho_0^{M_{2,k}},\,
 		\vm_0, \, \vm_0^{M_{1,k}}, \, \vm_0^{M_{2,k}},\,
 		E_0, \, E_0^{M_{1,k}}, \, E_0^{M_{2,k}},
 		\\
 		& 
 		{\cb \varrho^{M_{1,k}}, \, \varrho^{M_{2,k}}, \, 
 		\vm^{M_{1,k}}, \, \vm^{M_{2,k}}, \, 
 		E^{M_{1,k}}, \, E^{M_{2,k}},} \, 
 		C^{\cb M_{1,k}},{\cb C^{M_{2,k}} }\Big),
 	\end{align*}
 	{\cb where we dropped all subscripts $n(M_{l,k})$, $l=1,2$, indicating the mesh resolution in space}.
  	The Prohorov theorem implies convergence in law of $\{U_k\}_{k\in\naturals}$ in $X$. Using the Skorokhod representation theorem, Section 3.1.1 of Ref.~\citenum{skorokhod1956LimitTheorems}, possibly going to a subsequence and relabeling the sequence, we find that there exist $\tilde U_k \colon [0,1] \to X$ and $\tilde U \colon [0,1] \to X$,
  	\begin{align*}
  		\tilde U_k &= \Big(
  		\tilde \varrho_0^{0,k}, \, \tilde \varrho_0^{1,k}, \, \tilde \varrho_0^{{2,k}},\,
  		\tilde \vm_0^{0,k}, \, \tilde \vm_0^{{1,k}}, \, \tilde \vm_0^{{2,k}},\,
  		\tilde E_0^{0,k}, \, \tilde E_0^{{1,k}}, \, \tilde E_0^{{2,k}},\,
  		\\
  		& \phantom{= \Big( \qquad \qquad \qquad}
  		\tilde \varrho^{{1,k}}, \, \tilde \varrho^{{2,k}}, \, 
  		\tilde \vm^{{1,k}}, \, \tilde \vm^{{2,k}}, \, 
  		\tilde E^{{1,k}}, \, \tilde E^{{2,k}}, \, 
  		{\cb \tilde C^{{1,k}}, \, \tilde C^{{2,k}}}
  		\Big),
  		\\
  		\tilde U &= \Big(
  		\tilde \varrho_0^{0}, \enspace \tilde \varrho_0^{1}, \enspace \tilde \varrho_0^{2},\enspace
  		\tilde \vm_0^0, \enspace \tilde \vm_0^{1}, \enspace \tilde \vm_0^{2},\enspace
  		\tilde E_0^0, \enspace \tilde E_0^{1}, \enspace \tilde E_0^{2},\enspace
  		\\
  		& \phantom{= \Big( \qquad \qquad \qquad}
  		\tilde \varrho^{1}, \enspace \tilde \varrho^{2},\enspace
  		\tilde \vm^{1}, \enspace \tilde \vm^{2}, \enspace
  		\tilde E^{1}, \enspace \tilde E^{2}, \enspace
  		{\cb \tilde C^1, \, \tilde C^2}
  		\Big),
  	\end{align*}
  such that  $\tilde U_k(\tilde \omega) \to  \tilde U(\tilde \omega)$ in $X$ as $k \to \infty$ for a.a.  $\tilde \omega \in [0,1]$ and 
  \begin{equation}\label{eq:equalityOfLaws}
  	\mathds{P}(U_k \in A) = \lambda(\tilde{U}_k \in A) \qquad \text{for all }k\in\naturals,\, A\in \mathcal{B}(X).
  \end{equation}
  Note that in the original sequence $U_k$, the initial data {\cb $\varrho_0$, $\vm_0$, and $E_0$} did not depend on $k$. The corresponding entries $$\tilde \varrho_0^{0,k},\, \tilde \vm_0^{0,k},\, \tilde E_0^{0,k}$$ of $\tilde U_k$ however might depend on $k$ and need not be equal to the corresponding entries of $\tilde U$. The transformation of, e.g.,\ $\varrho_0$ to the map $\tilde \varrho_0^{0,k}$ from $[0,1]$ to $L^\gamma(Q)$ might be different for each $k$.

{\cb	Admissibility of the initial data can be shown by including $\underline{\varrho}^M(\omega) =  \sum_{m=1}^M \ind{\Omega_m^M}(\omega) \underline{\varrho}(\omega_m^M)$ and the similar defined $\overline{E}^M$ and $\underline{s}^M$ in $U_k$. We skip these details and refer the extended version of this article~\cite{extendedPaper} for the full proof.}
  
  For $i\in \{1,\dots, \nu(M_{1,k})\}$, $j\in \{1,\dots, \nu(M_{2,k})\}$ let $\omega \in \Omega_i^{M_{1,k}} \cap \Omega_j^{M_{2,k}}$. Clearly, for every fixed $\omega$ the set $A\subset X$ given by,
  \begin{align*}
 & A \coloneqq 
  \Bigg\{ 
  \Big(
  x,\, \varrho_0^{M_{1,k}}(\omega), \, \varrho_0^{M_{2,k}}(\omega),\,
  y, \, \vm_0^{M_{1,k}}(\omega), \, \vm_0^{M_{2,k}}(\omega),\,
  z, \, E_0^{M_{1,k}}(\omega), 
 \\&  E_0^{M_{2,k}}(\omega),\,
  \varrho_{n(M_{1,k})}^{M_{1,k}}(\omega), \, \varrho^{M_{2,k}}_{n(M_{2,k})}(\omega), \,
  \vm^{M_{1,k}}_{n(M_{1,k})}(\omega), \, \vm^{M_{2,k}}_{n(M_{2,k})}(\omega), \, 
  E^{M_{1,k}}_{n(M_{1,k})}(\omega), \,
  \\
  &  E^{M_{2,k}}_{n(M_{2,k})}(\omega), \,
  	C^{\cb M_{1,k}}(\omega),\,{\cb C^{M_{2,k}} }(\omega)
  \Big)
  \setsep x\in L^\gamma,\, y\in L^{2\gamma/(\gamma+1)}, \, z\in L^1
	\Bigg\},
\end{align*}
is Borel measurable and thus there exists $\tilde{\Omega}^k_{i,j} = (\tilde U^k)^{-1}(A)\subset [0,1)$ with $\lambda(\tilde{\Omega}^k_{i,j}) = \mathds{P}(\cb U_k^{-1} U_k(\Omega_i^{M_{1,k}} \cap \Omega_j^{M_{2,k}}))$. In particular, all components of $\tilde U_k$ other than $\tilde \varrho_0^{0,k}$, $\tilde \vm_0^{0,k}$, $\tilde E_0^{0,k}$
are a.e. equal to piecewise constant functions. As a consequence, $(\tilde \varrho^{1,k}, \tilde \vm^{1,k}, \tilde E^{1,k})$ satisfy the numerical scheme (\ref{eq:detDiscretization}) for the initial data $(\tilde\varrho_0^{1,k}, \tilde\vm_0^{1,k}, \tilde E_0^{1,k})$   a.e. in  $[0,1]$. The analogous statement is true for $(\tilde \varrho^{2,k}, \tilde \vm^{2,k}, \tilde E^{2,k}).$
  
Now, we fix $\tilde \omega \in [0,1]$ such that both $(\tilde \varrho^{1,k}, \tilde \vm^{1,k},\tilde E^{1,k})(\tilde \omega)$ and $(\tilde \varrho^{2,k}, \tilde \vm^{2,k}, \tilde E^{2,k})(\tilde \omega)$ are solutions of the numerical scheme (\ref{eq:detDiscretization}) and such that $\tilde U_k(\tilde\omega) \to \tilde U(\tilde\omega)$ in $X$.
As a consequence $(\tilde\varrho_0^{1,k}(\tilde\omega),\tilde\vm_0^{1,k}(\tilde\omega), \tilde E_0^{1,k}(\tilde\omega))$ and $(\tilde\varrho_0^{2,k}(\tilde\omega), \tilde\vm_0^{2,k}(\tilde\omega), \tilde E_0^{2,k}(\tilde\omega))$ converge in $L^\gamma(Q) \times L^{2\gamma/ (\gamma +1)}(Q; \reals^d)\times L^1(Q)$. Moreover, $\{\cb \tilde C^{l,k}(\tilde \omega)\}_{k\in\naturals}$ converges and thus is bounded. 
{\cb By Theorem~\ref{thm:deterministic_convergence}}
the limits $(\tilde{\varrho}^\ell$, $\tilde{\vm}^\ell$, $\tilde{E}^\ell)$  are  strong solutions of the Euler equations corresponding to the initial data $(\tilde{\varrho}_0^\ell$, $\tilde{\vm}_0^\ell$, $\tilde{E}_0^\ell)$ for $\ell=1,2$.
More precisely, for $k \to \infty$
\begin{eqnarray*}
&& (\tilde \varrho^{\ell,k}, \tilde \vm^{\ell,k},\tilde E^{\ell,k})(\tilde \omega)  \to    
 (\tilde{\varrho}^\ell, \tilde{\vm}^\ell, \tilde{E}^\ell)(\tilde \omega) \\
&&\mbox{ in } L^q((0,T); L^\gamma(Q) \times L^{2 \gamma/(\gamma + 1)}(Q; \Bbb R^d)
\times L^1 (Q)), \  1 \leq q < \infty.
\end{eqnarray*}
As the strong solution is unique, see Theorem\ref{lem:uniquenessStrongSolutions}, we only need to show that the initial data  $\tilde f^{1}(0,\cdot)$ and $\tilde f^{2}(0,\cdot)$ coincide for $f=\varrho, \, \vm,\, E$ to apply the Gyöngy-Krylov lemma, Lemma 1.1 of Ref.~\citenum{gyoengy1996existence}.

From now on we denote by $f_0$ either $\varrho_0$, $\vm_0$ or $E_0$.
Note that $f_0^{M_{1,k}}$ and $f_0^{M_{2,k}}$ converge $\mathds{P}$-a.s. to $f_0$ in the respective $L^p$ space. In particular, by Egorov's theorem, see Theorem 5.1.4 of Ref.~\citenum{malliavin1995}, there exist sets $A_k\in\mathcal{B}(\Omega)$ and an increasing sequence $\{n_k\}_{k\in\naturals}$ such that $\mathds{P}(A_k)< 2^{-k}$ and $\|f_0(\omega) - f_0^{M_{n_k,\ell}}(\omega)\|_{L^p} < 1/k$ for all $\omega\in A_k$ and  $\ell=1,2$.  In particular, after possibly going to subsequences and relabeling them, we find that 
\begin{equation*}
	\mathds{P}\left( \bigcup_m \bigcap_{k\geq m}\bigg\{ \|\tilde f_0^{\ell,k} -f_0^{0,k} \|<\epsilon  \bigg\}   \right) = 1, \quad \ell=1,2.
\end{equation*}
Here we have used that the laws of $\tilde U_k$ and $U_k$ coincide. Thus, the Gyöngy-Krylov lemma implies that 
\begin{eqnarray*}
&&(\varrho_{n(M)}^{M},  \vm_{n(M)}^{M}, E_{n(M)}^{M})  \to (\varrho, \vm, E) \\
&&\mbox{ in } L^q((0,T); L^\gamma(Q)\times L^{2\gamma/(\gamma+1)}(Q;\reals^d) \times L^1(Q)), \ 1 \leq q < \infty,
\end{eqnarray*}
for $M \to \infty, n(M) \to \infty $ in probability.
Applying again the Skorokhod theorem, but now for the sequence 
\begin{equation*}
	\Bigg \{\varrho_0, \varrho_0^M, \varrho_{n(M)}^{M}, \varrho, 
	\vm_0, \vm_0^M, \vm_{n(M)}^{M}, \vm, 
	E_0, E_0^M, E_{n(M)}^{M}, E,
	C^M \Bigg\}_{M = 1}^\infty 
\end{equation*}
we find that $(\varrho,\vm, E)$ is the random strong solution of the Euler equations with the initial data  $(\varrho_0, \vm_0, E_0)$. This concludes the proof.
\end{proof}

\section*{Acknowledgement}
This work was supported by the Gutenberg Research College and by
		the Deutsche Forschungsgemeinschaft (DFG, German Research Foundation) 
		project number 525853336 -- SPP 2410 ``Hyperbolic Balance Laws: Complexity, Scales and Randomness''.
		The authors also thank the  Mainz Institute of Multiscale Modelling for supporting their research.

\bibliographystyle{ws-procs9x6}

\end{document}